\documentclass[12pt]{article}

\usepackage{amsmath, amssymb, latexsym, amsthm}
\usepackage{fullpage}
\usepackage{color}
\usepackage{tikz}
\usepackage{graphicx}
\usepackage{stmaryrd,amsbsy,enumerate}
\usepackage{hyperref}
\usepackage{mathscinet}
\usetikzlibrary{calc}
\usepackage{cancel}
\usepackage[mathlines]{lineno}
\usepackage{url}

\newtheorem{theorem}{Theorem} 
\newtheorem{theorem*}{Theorem}
\newtheorem{proposition}[theorem]{Proposition}

\newtheorem{claim}[theorem]{Claim}

\renewcommand{\epsilon}{\varepsilon}

\theoremstyle{definition}

\theoremstyle{remark}

\newcommand{\URL}{\url{http://orion.math.iastate.edu/lidicky/pub/c5/}}

\newcounter{arxiv}

\begin{document}
\setcounter{arxiv}{1}

\title{Maximum density of an induced $5$-cycle is achieved by an iterated blow-up of  a $5$-cycle}
\author{
J\'{o}zsef Balogh\thanks{ Department of Mathematics, University of Illinois, Urbana, IL 61801, USA and Bolyai Institute, University of Szeged, Szeged, Hungary E-mail: {\tt jobal@math.uiuc.edu}. Research is partially supported by Simons Fellowship, NSF CAREER Grant DMS-0745185, Arnold O. Beckman Research Award (UIUC Campus Research Board 13039) and Marie Curie FP7-PEOPLE-2012-IIF 327763.} \and
Ping Hu\thanks{Department of Mathematics, University of Illinois, Urbana, IL 61801, USA and University of Warwick, UK E-mail: {\tt pinghu1@math.uiuc.edu}.}  \and
Bernard Lidick\'{y}\thanks{Department of Mathematics, Iowa State University, Ames, IA, E-mail: {\tt lidicky@iastate.edu}.}
\and
Florian Pfender\thanks{Department of Mathematical and Statistical Sciences, University of Colorado Denver, E-mail: {\tt 
Florian.Pfender@ucdenver.edu}. Research is partially supported by a collaboration grant from the Simons Foundation.} 
}

\newcommand{\Xmin}{0.19816}
\newcommand{\Xmax}{0.20184}
\newcommand{\Xzeromax}{0.0026}
\newcommand{\funky}{0.000011}
\newcommand{\funkyDeg}{0.0132}
\newcommand{\XzerofunkyDeg}{0.081}
\maketitle

\begin{abstract}
Let $C(n)$ denote the maximum number of induced copies of $5$-cycles in graphs
on $n$ vertices.
For $n$ large enough, we show that
$C(n)=a\cdot b\cdot c \cdot d \cdot e + C(a)+C(b)+C(c)+C(d)+C(e)$,
where $a+b+c+d+e = n$ and $a,b,c,d,e$ are as equal as possible.

Moreover, if $n$ is a power of 5, we show that the unique graph  on $n$ vertices
maximizing the number of induced $5$-cycles is an iterated blow-up of a 5-cycle.

The proof uses flag algebra computations and stability methods. 
\end{abstract}

\section{Introduction}

In 1975, Pippinger and Golumbic~\cite{PippengerGolumbic1975} conjectured that
in graphs the maximum induced density of a $k$-cycle is $k!/(k^k - k)$ when $k \geq 5$.
In this paper we solve their conjecture for $k=5$. In addition, we also show that
the extremal limit object is unique.
The problem of maximizing the induced density of $C_5$ is also posted
on \url{http://flagmatic.org} as one of the problems where the plain flag algebra
method was applied but failed to provide an exact result.
It was also mentioned by Razborov~\cite{RazborovIMA}. 

Problems of maximizing the number of induced copies of a fixed small graph $H$
have attracted a lot of attention recently~\cite{EvanLinial2013,Hirst2014,flagmatic}.
For a list of other results on this so called inducibility of small graphs
of order up to $5$, see the work of Even-Zohar and Linial~\cite{EvanLinial2013}.

{\ifnum\thearxiv=1
In this paper, we use a method that we originally developed for maximizing the
number of rainbow triangles in $3$-edge-colored complete graphs~\cite{BaloghHLPVY:2014+}.
However, the application of the method to the $C_5$ problem is less  technical,
and therefore this paper is a more accessible exposition of this new method.}{\fi}

Denote the $(k-1)$-times iterated blow-up of $C_5$ by $C_5^{k\times}$, 
see Figure~\ref{fig-construction}.
Let $\mathcal{G}_n$ be the set of all graphs on $n$ vertices, and
denote by $C(G)$ the number of induced copies of $C_5$ in a graph $G$. Define
\[
C(n) = \max_{G\in \mathcal{G}_n} C(G).
\]
We say a graph $G\in \mathcal{G}_n$ is \emph{extremal} if $C(G) = C(n)$. Notice that, since $C_5$ is a self-complementary graph, $G$ is extremal
if and only if its complement is extremal. If $n$ is a power of $5$, we can exactly determine the unique extremal graph and thus $C(n)$.

\begin{theorem}\label{thmC55k}
For $k\ge 1$, the unique extremal graph in $\mathcal{G}_{5^k}$
is $C_5^{k\times}$.
\end{theorem}

\begin{figure}[ht]
\begin{center}
\includegraphics[scale=0.8]{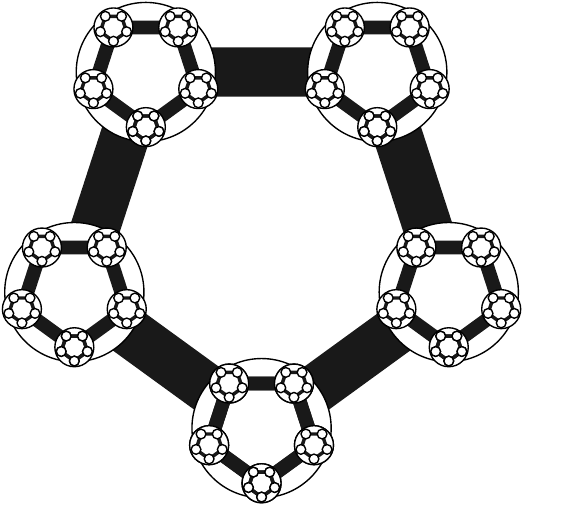}
\end{center}
\caption{The graph $C_5^{k\times}$ maximizes the number of induced $C_5$s.}\label{fig-construction}
\end{figure}

To prove Theorem~\ref{thmC55k}, we first prove the following theorem. Note that this theorem is sufficient to determine the unique limit object
(the graphon) maximizing the density of induced copies of $C_5$.
\begin{theorem}\label{thmrecurs}
There exists $n_0$ such that for every $n \geq n_0$
\[
C(n) = a\cdot b \cdot c \cdot d \cdot e + C(a) + C(b) + C(c) + C(d)+ C(e),
\]
where $a+b+c+d+e = n$ and $a,b,c,d,e$ are as equal as possible.

Moreover, if $G\in \mathcal{G}_n$ is an extremal graph, then $V(G)$ can be partitioned into five sets $X_1, X_2, X_3, X_4$ and $X_5$ of sizes $a, b, c, d$ and $e$ respectively, such that 
for 
$1\le i< j\le 5$ and $x_i\in X_i$, $x_j\in X_j$, we have $x_ix_j\in E(G)$ if and only if $j-i\in\{1,4\}$.
\end{theorem}

In the next section, we give a brief overview of our method,
in Section~\ref{sec:thmrecurs} we prove Theorem~\ref{thmrecurs},
and in Section~\ref{sec:thmC55k} we prove Theorem~\ref{thmC55k}.

\section{Method and Flag Algebras}

Our method relies on the theory of flag algebras  developed by Razborov~\cite{Razborov:2007}.
Flag algebras can be used as a general tool to attack problems from extremal combinatorics.
Flag algebras were used for a wide range of problems,
for example 
the Caccetta-H\"aggkvist conjecture~\cite{HladkyKN:2009,RazborovCH:2011},
Tur\'an-type problems in graphs~\cite{DasHMNS:2012,Grzesik:2011,Hatami:2011,PikhurkoV:2013,Razborov:2008,Reiher:2012,Sperfeld:2011}, 
$3$-graphs~\cite{Falgas:2012,GlebovKV:2013} 
and hypercubes~\cite{Baber:2012,BaloghHLL:2014},
extremal problems in a colored environment~\cite{BaberT:2013,BaloghHLPVY:2014+,CummingsKPSTY:2012}, 
and also to problems in geometry~\cite{Kral:2011} or extremal theory of permutations~\cite{BaloghHLPUV:2014}.
For more details on these applications, see a recent survey of Razborov~\cite{Razborov13}.

A typical application of the so called {\em plain flag algebra method} provides a bound on densities of substructures.
In some cases the bound is sharp, which happens most often when the extremal construction is `clean', 
for example a simple blow-up of a small graph, replacing each vertex by a large independent set. 
Obtaining an exact result from the sharp bound usually consists of first
bounding the densities of some small substructures by $o(1)$,
which can be read off from the flag algebra computation.
Forbidding these structures can yield a lot of structure of the extremal structure.
Finally, stability arguments are used to extract the precise extremal structure.

Simple blow-ups of small graphs appear very often as extremal graphs, in fact there are large families of graphs whose extremal graphs
for the inducibility are of this type, see Hatami, Hirst and Norin~\cite{Hatami2014}.
However, there are also many problems where the extremal construction is an iterated blow-up as shown by Pikhurko~\cite{OlegItterative}.

For our problem, the conjectured extremal graph has an iterated structure,
for which it is rare to obtain the precise density from plain flag algebra
computations alone. One such rare example is the problem to determine the inducibility of small out-stars
in oriented graphs~\cite{Falgas:2012} (note that the problem of inducibility of all out-stars
was recently solved by Huang~\cite{Huang:2014} using different techniques).
Hladk\'y, Kr\'a\soft{l} and Norin announced that they found the inducibility
 of the oriented path of length 2, which also has an
iterated extremal construction, via a flag algebra method. 
{\ifnum\thearxiv=1 Other than these two examples and~\cite{BaloghHLPVY:2014+},}{\fi
In~\cite{BaloghHLPVY:2014+} we determine the iterated extremal construction maximizing
the number of rainbow triangles in $3$-edge-colored complete graphs.
Other than these three examples,}
we are not aware of any applications of flag algebras which completely determined an iterative structure.

For our question, a direct application of the plain method gives 
an upper bound on the limit value and shows that 
$\lim_{n\to\infty}C(n)/{n\choose 5} < 0.03846157$,
which is slightly more than the density of $C_5$ in the conjectured extremal
construction, which is $\frac{1}{26}\approx  0.03846154$.
This difference may appear very small, but the bounds on densities of subgraphs
not appearing in the extremal structure are too weak to allow the standard methods to work.

Instead, we use flag algebras to find bounds on densities of other subgraphs, which appear with fairly
high density in the extremal graph. This enables us to better control the slight lack of performance
of the flag algebra bounds as these small errors have a weaker relative effect on larger densities.

\section{Proof of Theorem~\ref{thmrecurs}}\label{sec:thmrecurs}

In our proofs we consider densities of $7$-vertex subgraphs. 
Guided by their prevalence in the conjectured extremal graph, the following two types of graphs will play an important role.
We call a graph $C22111$ if it can be obtained from $C_5$ by duplicating two vertices.
We call a graph $C31111$ if it can be obtained from $C_5$ by tripling one vertex.
The edges between the original vertices and their copies are not specified, and there are two complementary types of $C22111$, depending 
on the adjacency of the two doubled vertices in $C_5$. Technically, $C22111$ and $C31111$ denote collections of several graphs. 
Examples of $C22111$ and $C31111$ are depicted in Figure~\ref{fig-conf}.
We slightly abuse notation by using $C22111$ and $C31111$ also to denote the densities of these graphs, i.e., the probability that randomly chosen $7$ vertices induce the appropriate $7$-vertex blow-up of $C_5$.
Moreover, for a set of vertices $Z$ we denote by $C22111(Z)$ and $C31111(Z)$ the densities
of $C22111$ and $C31111$ containing $Z$, i.e., for a graph $G$ on $n$ vertices, $C22111(Z)$ ($C31111(Z)$) is the number of $C22111$($C31111$) containing $Z$ divided by  $\binom{n-|Z|}{7-|Z|}$.

\begin{figure}[ht]
\begin{center}
\includegraphics[page=1]{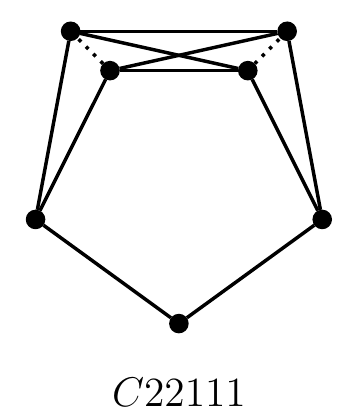}
\includegraphics[page=2]{fig-conf}
\includegraphics[page=3]{fig-conf}
\end{center}
\caption{Sketches of $C22111$ and $C31111$.}\label{fig-conf}
\end{figure}

\noindent
We start with the following statement. 

\begin{proposition} \label{prop:flag}  
There exists $n_0$ such that every extremal graph $G$ on at least  $n_0$ vertices satisfies:
\begin{align}
  C_5 &< 0.03846157;  \nonumber\\
4\cdot C22111-11.94 \cdot C31111 &\ge \frac{1349894760355389179787709186391}{420000000000000000000000000000000} + o(1)>  0.003214. \label{diff}
\end{align}
\end{proposition}
\begin{proof}
This follows from a standard application of the plain flag algebra method. 
The first inequality was obtained by Flagmatic~\cite{flagmatic}, which also provides
the corresponding certificate.
For the second inequality, we minimize the left side
with the extra constraint that $C_5\ge \frac1{26}$. 
We performed the computation on $7$ vertices since the resulting bound
was sufficient and rounding the solution is easier on $7$ vertices than
on $8$. 
For certificates, see \URL.
\end{proof}

The expressions from Proposition~\ref{prop:flag} compare to the following limiting values in the iterated blow-up $C_5^{k\times}$, where $k\to\infty$:
\begin{align*}
  C_5 &= \frac{1}{26} \approx 0.03846154; &
4\cdot C22111-11.94 \cdot C31111 &=  4\cdot\frac{5}{31}-11.94 \cdot\frac{5}{93} \approx 0.0032258.
\end{align*}
Notice that in the iterated blow-up of $C_5$, in the limit $4\cdot C22111-12 \cdot C31111=0$. For our method to work,
we need a lower bound greater than zero. On the other hand, computational experiments convinced us that the method works best if the 
bound is only slightly above zero, where a suitable factor is again determined by computations. 

Let $G$ be an extremal graph on $n$ vertices, where $n$ is sufficiently large to apply Proposition~\ref{prop:flag}.
Denote the set of all induced $C_5$s in $G$ by $\mathcal{Z}$.
We assume that $a \in \mathbb{R}$ and $Z=z_1z_2z_3z_4z_5$ is an induced $C_5$ maximizing $C22111(Z) - a\cdot C31111(Z) $.
Then
\begin{align}
 & \left(C22111(Z) - a\cdot C31111(Z)\right)\binom{n-5}{2} \nonumber 
\geq \frac{1}{|\mathcal Z|} \sum_{Y \in \mathcal Z}\left(C22111(Y) - a\cdot C31111(Y)\right)\binom{n-5}{2} \nonumber\\
&=~ \frac{\left( 4\cdot C22111 - 3a\cdot C31111\right)\binom{n}{7}}{C_5\binom{n}{5}} \nonumber \
=\frac{\tfrac{4}{21}C22111 - \tfrac{a}{7}C31111}{C_5} \binom{n-5}{2}. \nonumber
\end{align}

As mentioned above, computations indicate that we get the most useful bounds if $C22111(Z) - a\cdot C31111(Z) $ is close but not too close to $0$.
Using \eqref{diff} and letting $a=3.98$, we get
\begin{align}
C22111(Z) - 3.98\cdot C31111(Z) 
> 0.003979. \label{main2}
\end{align}

For $1\le i\le 5$, we define sets of vertices $Z_i$ which look like $z_i$
to the other vertices of $Z$. Formally,
\[
Z_i:=\{v\in V(G):G[(Z\setminus z_i)\cup v] \cong C_5\}\text{ for } 1\le i\le 5.
\]
Note that $Z_i\cap Z_j=\emptyset$ for $i\ne j$. 
We call a pair $v_iv_j$ {\em funky}, if $v_iv_j$ is an edge while $z_iz_j$ is not an edge or vice versa, where $v_i\in Z_i$, $v_j\in Z_j$, $1\le i<j\le 5$.
In other words, $G[Z\cup\{v_i,v_j\}]\ncong C22111$, i.e., every funky pair destroys a potential copy of $C22111(Z)$. Denote by $E_f$ the set of funky pairs. With this notation, \eqref{main2} implies that
\[
\sum_{1\le i<j\le 5}|Z_i||Z_{j}|-|E_f|-3.98\sum_{i\in[5]}|Z_i|^2/2 >0.003979\binom{n-5}{2}.
\]
For any choice of sets $X_i\subseteq Z_i$, where $i\in[5]$, let $X_0:=V(G)\setminus\bigcup X_i$. Let $f$ be the number of funky pairs not incident to vertices in $X_0$, divided by $n^2$ for normalization, and denote $x_i =\tfrac1n|X_i|$ for $i \in \{0,\ldots,5\}$. Choose the $X_i$ (possibly $X_i=Z_i$) such that the left hand side in
\begin{align}
2\sum_{1\le i<j\le 5}x_ix_{j}-2f-3.98\sum_{i\in [5]}x_i^2 > 0.003979\label{main3}
\end{align}
is maximized. 
In order to simplify notation, we use $X_{i+5}=X_{i}$ and $x_{i+5} = x_i$ for all $i \geq 1$.

\begin{claim}
The following equations are satisfied:
\begin{align}
\Xmin< x_i &<\Xmax \quad\text{for}\quad i\in[5]; \label{xbound}\\
x_0&<\Xzeromax;\label{Xzeromax}\\
f &<  \funky.\label{funky}
\end{align}
\end{claim}
\begin{proof}
To obtain \eqref{xbound}--\eqref{funky}, we need to solve four quadratic programs.
The objectives are to minimize $x_1$, maximize $x_1$, maximize $x_0$, and to maximize $f$, respectively.
The constraints are \eqref{main3} and $\sum_{i=0}^5 x_i = 1$ in all four cases.
By symmetry, bounds for $x_1$ apply also for $x_2$, $x_3$, $x_4$, and $x_5$.

Here we describe the process of obtaining the lower bound on $x_1$ in \eqref{xbound}. 
We need to solve the following program $(P)$:
\[
(P)  \begin{cases}
\text{minimize} & x_1\\
\text{subject to}
& \sum_{i=0}^5 x_i = 1, \\
& 2\sum_{1\le i<j\le 5}x_ix_{j}-2f-3.98\sum_{i\in [5]}x_i^2 > 0.003979,\\
& x_i \geq 0 \text{ for } i \in \{0,1,\ldots,5\}. \\
\end{cases}
\]
We claim that if $(P)$ has a feasible solution $S$,
then there exists a feasible solution $S'$ of $(P)$ where 
\begin{align*}
S'(x_1) = S(x_1),\quad S'(f) &= 0, \quad S'(x_0)=S(x_0), \\
S'(x_2) = S'(x_3) = S'(x_4)&=S'(x_5) =  \frac{1}{4}\big(1 - S(x_1)-S(x_0)\big).
\end{align*}

Since $x_2$, $x_3$, $x_4$ and $x_5$ appear only in constraints, we only need to
check if \eqref{main3} is satisfied.
The left hand side of \eqref{main3}  can be rewritten as
\begin{align*}
&2x_1\sum_{2\le i<j\le 5}x_i + 2\sum_{2\le i<j\le 5}x_ix_j-3.98\sum_{1\le i<j\le 5} x_i^2 -2f  \\
&= ~ 2x_1\sum_{2\le i<j\le 5}x_i  -\sum_{2\le i<j\le 5} (x_i-x_j)^2 - 0.98\sum_{2\le i<j\le 5} x_i^2 -3.98x_1^2-2f.
\end{align*}

Note that  the term $\sum_{2\le i<j\le 5} (x_i-x_j)^2$ is minimized if $x_i=x_j$ for
all $i,j \in \{2,3,4,5\}$. 
The term $x_2^2+x_3^2+x_4^2+x_5^2$, subject to $x_2+x_3+x_4+x_5$ being a constant, is also minimized
if $x_i=x_j$ for all $i,j \in \{2,3,4,5\}$.
Since $f \geq 0$, the term $2f$ is minimized when $f=0$.
Hence \eqref{main3} is satisfied by $S'$ and we can add the constraints $x_2=x_3=x_4=x_5$ and $f=0$ to bound $x_1$. 
The resulting program $(P')$ is
\[
(P')  \begin{cases}
\text{minimize} & x_1\\
\text{subject to}
& x_0+x_1+4y= 1, \\
& 8x_1y  - 0.98\cdot4 y^2 -3.98x_1^2 \geq 0.003979,\\
& x_0,x_1,y \geq 0. \\
\end{cases}
\]
We solve $(P')$ using Lagrange multipliers. We delegate the work to Sage~\cite{sage} and 
we provide the Sage script at \URL.
Finding an upper bound on $x_1$ is done by changing the objective to maximization.

Similarly, we can set $x_1=x_2=x_3=x_4=x_5 = 1/5$ to get an upper bound on $f$.
We can set $f=0$ and $x_1=x_2=x_3=x_4=x_5 = (1-x_0)/5$ to get an upper bound on $x_0$. 
We omit the details. Sage scripts for solving the resulting programs are provided at  \URL.
\end{proof}

For any vertex $v\in X_i,i\in[5]$ we use $d_f(v)$ to denote the number of funky pairs from $v$ to $(X_1\cup X_2\cup X_3\cup X_4\cup X_5)\setminus X_i$ after normalizing by $n$. 
If we move $v$ from $X_1$ to $X_0$, then the left hand side of \eqref{main3} will decrease by
\[
\frac{1}{n}\left(2(x_2+x_3+x_4+x_5) - 2d_f(v)-2\cdot 3.98 \cdot x_1 + o(1)\right).
\]
If this quantity was negative, then  the left hand side of \eqref{main3} could be increased by moving $v$ to $X_0$, contradicting our choice of  $X_i$.
This together with~\eqref{xbound} implies that
\begin{align}
 d_f(v) \leq x_2+x_3+x_4+x_5-3.98\cdot x_1+o(1) \le 1-4.98\cdot x_1+o(1)\le \funkyDeg.\label{maxfunky}
\end{align}
Symmetric statements hold also for  every vertex $v\in X_2 \cup X_3 \cup X_4 \cup X_5$.

\begin{claim}\label{nofunky}
There are no funky pairs.
\end{claim}

\begin{proof}
Assume there is a funky pair $uv$.
By symmetry, we only need to consider two cases, either $u\in X_1,v\in X_2$ or $u\in X_1, v\in X_3$. In fact, it is sufficient to
check the case where $u\in X_1$ and $v\in X_2$, so $uv$ is not an edge. The other case then follows from considering the complement of $G$. 

Let $G'$ be a graph obtained from $G$ by adding the edge $uv$, i.e., changing $uv$ to be not funky.
We compare the number of induced $C_5$s containing $\{u,v\}$  in $G$ 
and  in $G'$.
In $G'$, there are at least 
\[
\left[x_3x_4x_5-(d_f(u)+d_f(v))\max\{x_3x_4,x_3x_5,x_4x_5\}-f\cdot\max\{x_3,x_4,x_5\}\right]n^3
\]
induced $C_5$s containing $uv$, since we can pick one vertex from each of $X_3,X_4,X_5$ to form an induced $C_5$ as long as none of the resulting nine pairs is funky.

Now we count the number of induced $C_5$s in $G$ containing $\{u,v\}$. The number of such $C_5$s which contain vertices from $X_0$ is upper bounded by $x_0n^3/2$. 
Next we count the number of such $C_5$s avoiding $X_0$.  
Observe that there are no $C_5$s avoiding $X_0$ in which $uv$ is the only funky pair.

The number of $C_5$s containing another funky pair $u'v'$ with $\{u,v\}\cap\{u',v'\}=\emptyset$ can be upper bounded by $fn^3$.
We are left to count $C_5$s where the other funky pairs contain $u$ or $v$.
The number of $C_5$s containing at least two vertices other than $u$ and $v$ which
are in funky pairs can be upper bounded by $(d_f(u)^2/2+d_f(v)^2/2+d_f(u)d_f(v))n^3$.

It remains to count only $C_5$s containing exactly one vertex $w$ where
$uw$ and $vw$ are the options for funky pairs. The number of choices of $w$ is at most $(d_f(u)+d_f(v))n$.
As $\{u,v,w\}$ is in an induced $C_5$, the set $\{u,v,w\}$ induces a path in either $G$ or the complement of $G$.
Let the middle vertex of that path be in $X_i$.
If $G[\{u,v,w\}]$ is a path, then the remaining two vertices of a $C_5$ cannot be in $X_{i+1}\cup X_{i+4}$. 
If $G[\{u,v,w\}]$ is the complement of a path, then the remaining two vertices cannot be in $X_{i+2}\cup X_{i+3}$.
Hence the remaining two vertices of a $C_5$ containing $\{u,v,w\}$ can be chosen from at most $3n \max\{x_i\}$ vertices.
This gives an upper bound of $(d_f(u)+d_f(v))n\binom{3n \max\{x_i\}}{2}$ on the number of such $C_5$s.

Now we compare the number of induced $C_5$s containing $uv$ in $G$ and in $G'$ . 
We use $x_{max}$ and $x_{min}$ to denote the upper and lower bound respectively from \eqref{xbound}, use $d_f$ to denote the upper bound on $d_f(u)$ and $d_f(v)$ from \eqref{maxfunky}, and also use bounds from \eqref{Xzeromax} and \eqref{funky}.
The number of $C_5$s containing $uv$ divided by $n^3$ is
\begin{align*}
\mbox{in }G: &\le x_0/2 + f+2d_f^2+9d_fx_{max}^2 \le 0.0065;\\
\mbox{in }G':&\ge (x_{min}-2d_f)x_{min}^2-fx_{max} \ge 0.0067.
\end{align*}
This contradicts the extremality of $G$.
\end{proof}

Next, we want to show that $X_0=\emptyset$. For this, suppose that there exists an $x\in X_0$. We will add $x$ to one of the $X_i$, $i \in [5]$ such that $d_f(x)$ is minimal. 
By symmetry, we may assume that $x$ is added to $X_1$. Note that adding a single vertex to $X_1$ does not change any of the density bounds we used above by more than $o(1)$.

\begin{claim}\label{X0funky}
For every $x \in X_0$, if $x$ is added to $X_1$ then $d_f(x)\geq \XzerofunkyDeg$.
\end{claim}
\begin{proof}
Let $xw$ be a funky pair, where $w \in X_2$. The case where $w \in X_3$ can be argued the same way by considering the complement of $G$.
Let $G'$ be obtained from $G$ by adding the edge $xw$.
Since $G$ is extremal, 
we have $C(G')\le C(G)$.
The following analysis is similar  to the proof of Claim~\ref{nofunky}, however, we can say a bit more since every funky pair contains $x$.

First we count induced $C_5$s containing $xw$ in $G$.
The number of induced $C_5$s containing $xw$ and other vertices from $X_0$ is easily bounded from above by $x_0n^3/2$. 

Let $F$ be an induced $C_5$ in $G$ containing $xw$ and avoiding $X_0\setminus\{x\}$.
Since all funky pairs contain $x$, $F-x$ is an induced path $p_0p_1p_2p_3$ without funky pairs.
Either $p_j \in X_2$ for all $j \in \{0,1,2,3\}$ or there is an $i \in \{1,2,3,4,5\}$ 
such that $p_j \in X_{i+j}$ for all $j \in \{0,1,2,3\}$. 
The first case is depicted in Figure~\ref{fig-funky}(a). Consider now the second case.
If $i \in \{2,3,4\}$, then $xp_0p_1p_2p_3$ does not satisfy the definition of $F$.
Hence $i \in \{1,5\}$ and the possible $C_5$s are depicted in Figure~\ref{fig-funky}(b) and (c).
In each of the three cases, $F$ contains exactly two funky pairs, $xw$ and $xy$. 
The location of $y$ entirely determines the location of $F-x$. 
Hence the number of induced $C_5$s containing $xw$ is at most $d_f(x)x_{max}^2n^3$. 

In $G'$, there are at least $\big(x_3x_4x_5-d_f(x)\max\{x_3x_4,x_3x_5,x_4x_5\}\big)n^3$
 induced $C_5$s containing $xw$.
We obtain
\begin{align*}
C(G)/n^3&\le d_f(x)x_{max}^2+x_0/2 & &\text{and}  &
C(G')/n^3&\ge (x_{min}-d_f(x))x_{min}^2.
\end{align*}
Since $C(G')\le C(G)$, we have
\[
(x_{min}-d_f(x))x_{min}^2\le d_f(x)x_{max}^2+x_0/2,
\]
which together with \eqref{xbound} and \eqref{Xzeromax} gives $d_f(x)\ge \XzerofunkyDeg$.
\end{proof}

\begin{figure}
\begin{center}
\includegraphics[page=1,scale=0.8]{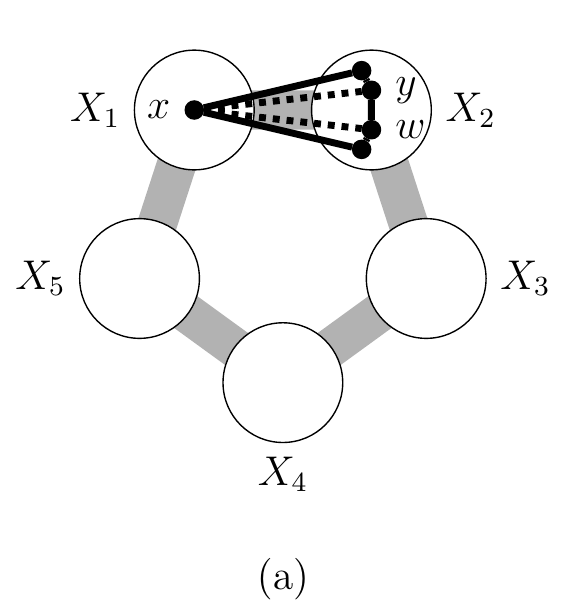}
\includegraphics[page=2,scale=0.8]{fig-funky}
\includegraphics[page=3,scale=0.8]{fig-funky}
\end{center}
\caption{Possible $C_5$s with funky pair $xw$. They all have exactly one other funky pair $xy$.\label{fig-funky}}
\end{figure}

\begin{claim}\label{uniformVertex}
Every vertex of the extremal graph $G$ is in at least $(1/26+o(1))\binom{n}{4}\approx 0.001602564 n^4$ induced $C_5$s.
\end{claim}

\begin{proof}

For every vertex $u\in V(G)$, denote by $C_5^u$ the number of $C_5$s in $G$ containing $u$.
For any two vertices $u,v\in V(G)$, we show that $C_5^u - C_5^v  < n^3$, which implies Claim~\ref{uniformVertex}.
Denote by $C_5^{uv}$ the number of $C_5$s in $G$ containing both $u$ and $v$.
A trivial bound is $C_5^{uv} \leq \binom{n-2}{3}$.

Let $G'$  be obtained from $G$ by deleting $v$ and 
duplicating $u$ to $u'$, i.e., for every vertex $x$ we add the edge $xu'$ iff $xu$ is an edge.
As $G$ is extremal we have
\begin{align*}
0\ge C(G')-C(G)\ge C_5^u - C_5^v  - C_5^{uv}  \ge  C_5^u - C_5^v- \binom{n-2}{3}.
\end{align*}
\end{proof}

\begin{claim}\label{X0empty}
The set $X_0$ is empty.
\end{claim}
\begin{proof}
Assume that there is an $x\in X_0$. We count $C_5^x$, the number  of induced $C_5$s containing $x$. Our goal is to show that $C_5^x$ is smaller than the value in
Claim~\ref{uniformVertex}, which is a contradiction.
Let $a_i n$ be the number of neighbors of $x$ in $X_i$ and $b_i n$ be the number of non-neighbors of $x$ in $X_i$ for $i \in \{0,1,2,3,4,5\}$. 

The number of $C_5$s where the other four vertices are in $X_1\cup X_2\cup X_3\cup X_4\cup X_5$
is upper bounded by
\[
\left(a_1b_2b_3a_4 + a_2b_3b_4a_5 + a_3b_4b_5a_1+a_4b_5b_1a_2+a_5b_1b_2a_3+\frac{1}{4}\sum_{i=1}^5a_i^2b_i^2\right)n^4.
\]
Moreover, we also need to include the $C_5$s containing vertices from $X_0$ in our bound,
which we do very generously by increasing all variables by $a_0$ or $b_0$.

Since $x_i = a_i+b_i$, we can use \eqref{xbound} for every $i \in [5]$ as constraints.
We also use Claim~\ref{X0funky} to obtain constraints since
it is possible to express $d_f(x)$ using $a_i$s and $b_i$s if $x$ is added to $X_j$ for all $i,j \in [5]$.

By combining the previous objective and constraints, we obtain the following program $(P)$, whose objective gives an upper bound on the number
of $C_5$s containing $x$ divided by $n^4$.
\[
(P)  \begin{cases}
\text{maximize}  
& \sum_{i=1}^5(a_i+a_0)(b_{i+1}+b_0)(b_{i+2}+b_0)(a_{i+3}+a_0)+\frac{1}{4}\sum_{i=1}^5a_i^2b_i^2\\
\text{subject to}
& \sum_{i=0}^5 (a_i+b_i)= 1, \\
& \Xmin\le a_i+b_i\le\Xmax  \text{ for } i \in \{1,2,3,4,5\}, \\
& a_0+b_0 \le \Xzeromax, \\
&    b_2+b_5+a_3+a_4\ge \XzerofunkyDeg,\\
 &   b_1+b_3+a_4+a_5\ge \XzerofunkyDeg,\\
  &  b_2+b_4+a_1+a_5 \ge \XzerofunkyDeg,\\
  &  b_3+b_5+a_1+a_2 \ge \XzerofunkyDeg,\\
  &  b_4+b_1+a_2+a_3 \ge \XzerofunkyDeg,\\
& a_i,b_i\geq 0   \text{ for } i \in \{0,1,2,3,4,5\}. \\
\end{cases}
\]

\newcommand{\XmaxRelax}{0.21}
\newcommand{\XzerofunkyDegRelax}{0.081}

Instead of solving $(P)$ we solve a slight relaxation $(P')$ with increased upper bounds on $a_i+b_i$,
which allows us to drop $a_0$ and $b_0$. Since the objective function is maximizing, we can claim
that $a_i+b_i$ is always as large as possible, which decreases the degrees of freedom.
\[
(P')  \begin{cases}
\text{maximize}  
& f=\sum_{i=1}^5a_ib_{i+1}b_{i+2}a_{i+3}+\frac{1}{4}\sum_{i=1}^5a_i^2b_i^2\\
\text{subject to}
&    a_i+b_i = \XmaxRelax  \text{ for } i \in \{1,2,3,4,5\}, \\
&    b_2+b_5+a_3+a_4\ge \XzerofunkyDeg,\\
 &   b_1+b_3+a_4+a_5\ge \XzerofunkyDeg,\\
  &  b_2+b_4+a_1+a_5 \ge \XzerofunkyDeg,\\
  &  b_3+b_5+a_1+a_2 \ge \XzerofunkyDeg,\\
  &  b_4+b_1+a_2+a_3 \ge \XzerofunkyDeg,\\
& a_i,b_i\geq 0   \text{ for } i \in \{1,2,3,4,5\}. \\
\end{cases}
\]
Note that the resulting program $(P')$ has only 5 degrees of freedom. 
We find an upper bound of the solution of $(P')$ by a brute force method. We discretize
the space of possible solutions, and bound the gradient of the target function to control the
behavior between the grid points. 

For solving $(P')$, we fix a constant $s$ which will correspond to the number of steps. For every
$a_i$ we check $s+1$ equally spaced values between $0$ and $\XmaxRelax$ that include
the boundaries.
By this we have a grid of $s^5$ boxes where every feasible solution of $(P')$, and hence also of $(P)$, is in one of the boxes.

Next we need to find the partial derivatives of $f$. Since $f$ is symmetric, we only check the partial
derivative with respect to $a_1$.
\[
\frac{\partial f}{\partial a_1} = b_{2}b_{3}a_{4} + a_{3}b_{4}b_{5} + \frac{1}{2}a_1b_1^2
\]
We want to find an upper bound on $\frac{\partial f}{\partial a_1}$. 
We can pick $0.21$ as an upper bound
on $a_i+b_i$. 
Hence we assume $a_1+b_1=a_3+b_3=a_4+b_4=b_{2}=b_5=0.21$ and we maximize 
\[
b_2b_{3}a_{4}+a_{3}b_{4}b_5= 0.21\left((0.21-a_3)a_4  + a_3(0.21-a_4)\right) = 0.21\left(0.21a_4  + 0.21a_3 -2a_3a_4\right).
\]
This is maximized if $a_3=0, a_4=0.21$ or $a_3=0.21, a_4=0$ and gives the value $0.21^3$.
Hence 
\[
\frac{1}{2}a_1b_1^2 =\frac42 a_1\cdot \frac{b_1}{2}\cdot \frac{b_1}{2}\le \frac{2( a_1+b_1)^3}{3^3} = \frac{2\cdot 0.21^3}{27}.
\]
The resulting upper bound is
\[
\frac{\partial f}{\partial a_1} \leq 0.21^3 + \frac{2\cdot 0.21^3}{27} < 0.001.
\]
Hence in a box with side length $t$ the value of $f$ cannot be bigger than the
value at a corner plus $5t/2\cdot0.001$. The factor  $5t/2$ comes from
the fact that the closest corner is in distance at most $t/2$ in each of the $5$ coordinates. 

If we set $s = 100$, we compute that the maximum over all grid points of $(P'')$ is less than  0.00157.
This can be checked by a computer program \texttt{mesh-opt.cpp} which computes the values at all grid points. 
With $t<0.21/s=0.0021$, we have $5t/2\cdot0.001<0.00001 $.
We conclude that $x$ is in less than $0.00158n^4$ induced $C_5$s which contradicts Claim~\ref{uniformVertex}.

Let us note that if we choose $s = 200$, we could conclude that $x$ is less than $0.00147 n^4$.
\end{proof}

We have just established the ``outside'' structure of $G$.
Observe that in this outside structure, an induced $C_5$ can appear only if it either intersects each of the classes 
in exactly one vertex, or if it lies completely inside one of the classes.
This implies that
\[C(n)=(x_1\cdot x_2\cdot x_3 \cdot x_4 \cdot x_5)n^5 + C(x_1n)+C(x_2n)+C(x_3n)+C(x_4n)+C(x_5n).\]
By averaging over all subgraphs of $G$ of order $n-1$, we can easily see that $C(n)\le C(n-1)$ for all $n$, so
\[
\ell:=\lim_{n\to\infty}\frac{C(n)}{{n\choose 5}}
\]
exists. Therefore, 
\[\ell+o(1)=5!\cdot x_1\cdot x_2\cdot x_3 \cdot x_4 \cdot x_5  + \ell (x_1^5+x_2^5+x_3^5+x_4^5+x_5^5),\]
which implies that $x_i=\frac15+o(1)$, and $\ell=\frac1{26}$, given the constraints on the $x_i$.

In order to prove Theorem~\ref{main2}, it remains  to show that in fact $|X_i|-|X_j|\le 1$ for all $i,j \in \{1,\ldots,5\}$.

\newcommand{\avg}{\mathrm{avg}}
\begin{claim}\label{clbalance}
For $n$ large enough, we have $|X_i|-|X_j|\le 1$ for all $i,j \in \{1,\ldots,5\}$.
\end{claim}
\begin{proof}
By symmetry, assume for contradiction that $|X_1|-|X_2|\ge 2$.
Let $v \in X_1$ where $C_5^v$ is minimized  over the vertices in $X_1$
and let $w \in X_2$ where $C_5^w$ is maximized  over the vertices in $X_2$.
As $G$ is extremal, $C_5^v+C_5^{vw}-C_5^w\ge 0$; otherwise, we can increase the number of $C_5$s 
by replacing $v$ by a copy of $w$.

Let $y_i:=|X_i|=x_in$. By the monotonicity of $\frac{C(n)}{n^5}$, we have
\[
\frac1{26}+o(1)\ge \frac{C(y_2)}{{y_2\choose 5}}\ge \frac{C(y_1)}{{y_1\choose 5}}\ge \frac1{26} -o(1).
\]
Therefore, using that $y_1-y_2\ge 2$,
\begin{align*}
C_5^v+C_5^{vw}-C_5^w&\le \frac{C(y_1)}{y_1}+y_2y_3y_4y_5+y_3y_4y_5-\frac{C(y_2)}{y_2}-y_1y_3y_4y_5\\
&= \frac{y_2C(y_1)-y_1C(y_2)}{y_1y_2}+(y_2-y_1+1)y_3y_4y_5\\
&\le \left( \frac1{26}+o(1)\right)\frac1{y_1y_2} \left( y_2{y_1\choose 5}-y_1{y_2\choose 5}\right)+(y_2-y_1+1)y_3y_4y_5\\
&\le \left( \frac1{26\cdot 5!}+o(1)\right) \left( y_1^4-y_2^4\right)+(y_2-y_1+1)y_3y_4y_5\\
&= \left( \frac1{26\cdot 5!}+o(1)\right) (y_1-y_2)\left( y_1^3+y_1^2y_2+y_1y_2^2+y_2^3\right)+(y_2-y_1+1)y_3y_4y_5\\
&= (y_1-y_2)\left(\left( \frac1{26\cdot 5!}+o(1)\right)  \frac{4n^3}{125}  - \frac{n^3}{125} \right)+\frac{(1+o(1))n^3}{125}\\
&\le \left( \frac2{26\cdot 5!}+o(1)\right)  \frac{4n^3}{125}  - \frac{(1+o(1))n^3}{125} 
<0,
\end{align*}
a contradiction.
\end{proof}

With this claim, the proof of Theorem~\ref{thmrecurs} is complete.

\section{Proof of Theorem~\ref{thmC55k}}\label{sec:thmC55k}

Theorem~\ref{thmC55k} is a consequence of Theorem~\ref{thmrecurs}.
The main proof idea is to take a minimal counterexample $G$ and show that some blow-up of $G$  contradicts Theorem~\ref{thmrecurs}.

\begin{proof}[Proof of Theorem~\ref{thmC55k}.]
Theorem~\ref{thmC55k} is easily seen to be true for $k=1$.
Suppose for a contradiction that there is a graph $G$ on $n=5^k$ vertices with $C(G)\ge C(C_5^{k\times})$ that
is not isomorphic to $C_5^{k\times}$, where $k\ge 2$ is minimal.
Let $n_0$ be the $n_0$ from the statement of Theorem~\ref{thmrecurs}.

We say that a graph $F$ of size $5m$ can be \emph{5-partitioned}, if $V(F)$ can be partitioned
into five sets  $X_1,X_2,X_3,X_4,X_5$ with $|X_i|=m$ for all $i \in [5]$ and for every 
$1 \leq i < j \leq 5$, every $x_i \in X_i$ and $x_j \in X_j$ are adjacent
if and only if $|i-j| \in \{1,4\}$. Notice that this is the structure described by Theorem~\ref{thmrecurs}. Hence if  $5m \geq n_0$, and $F$ is extremal then $F$ can be $5$-partitioned.

If $G$ can be $5$-partitioned, then $G$ is isomorphic to $C_5^{k\times}$ by the minimality of $k$, a contradiction.
Therefore, $G$ cannot be $5$-partitioned.

Let $H$ be an extremal graph on $5^\ell>n_0$ vertices.
Blowing up every vertex of $C_5^{k\times}$ by a factor of $5^\ell$, and inserting $H$ in every part, gives an extremal graph $G_1$
on $5^{k+\ell}$ vertices by $\ell$ applications of Theorem~\ref{thmrecurs}. On the other hand, the graph $G_2$ obtained by blowing up every vertex of $G$ by a factor of $5^\ell$, 
and inserting $H$ in every part, contains at least as many $C_5$s as $G_1$,
\begin{align*}
C(G_1) &= 5^k\cdot C(H) + C(C_5^{k\times}) \cdot (5^\ell)^5, &
C(G_2) &= 5^k\cdot C(H) + C(G) \cdot (5^\ell)^5,
\end{align*}
so $C(G_1) \leq C(G_2)$. Hence $G_2$ must also be extremal.
Therefore $G_2$ can be $5$-partitioned into five sets $X_1,X_2,X_3,X_4,X_5$ with $|X_i|=5^{k+\ell-1}$. In particular, two vertices in $G_2$
are in the same set $X_i$ if and only if their adjacency pattern agrees on more than half of the remaining vertices. But this implies that for every copy $H'$ of $H$ inserted into the blow up of $G$, all vertices of $H'$ are in the same $X_i$, and thus  the $5$-partition of $V(G_2)$ gives a $5$-partition of $V(G)$, a contradiction.
\end{proof}

\section*{Acknowledgement}
We would like to thank Jan Volec for fruitful discussions.

\bibliographystyle{abbrv}
\bibliography{refs.bib}

\end{document}